\documentclass[12 pt]{article}

\RequirePackage{amsthm,amsmath,amsfonts,amssymb}
\RequirePackage[authoryear]{natbib}
\RequirePackage[colorlinks,citecolor=blue,urlcolor=blue]{hyperref}
\RequirePackage{graphicx}

\usepackage[margin=1in]{geometry}
\usepackage{graphicx}
\usepackage{multirow}
\usepackage{mathrsfs}
\usepackage{bigints}
\usepackage{bbm}
\usepackage{xcolor}
\usepackage{verbatim}
\usepackage{caption}
\usepackage{subcaption}

\usepackage{natbib}
\usepackage{algorithm}
\usepackage{algpseudocode}

\usepackage{setspace}
\newtheorem{theorem}{Theorem}[section]

\newtheorem{example}{Example}[section]
\newtheorem{corollary}{Corollary}[example]

\usepackage{titling}
\title{Large Sample Properties of Higher Order Markov Models}
\author{Tuhin Majumder\\
    Department of Statistical Sciences, Wake Forest University\\
    and \\
    Donald E.K. Martin \\
    Department of Statistics, North Carolina State University\\
    and \\
    Soumendra N. Lahiri \\
    Department of Mathematics and Statistics, Washington University in St. Louis}

\date{\vspace{-5ex}}
\begin{document}

\maketitle

\begin{abstract}
We study large-sample properties of higher-order Markov chains on a finite alphabet $\Sigma$ when the order $m_n$ is allowed to grow with the sequence length $n$. By embedding the process into a first-order chain on $\Sigma^{m_n}$ and exploiting return-time decompositions, we establish a central limit theorem for additive functionals $\sum_{t}\! g_n(Y_t^{(n)})$ under natural ergodicity and sparsity conditions. The normalization involves the stationary return time to a suitably chosen state and accommodates triangular arrays with $m_n\!\to\!\infty$ and $m_n/n\!\to\!0$. We further illustrate the assumptions in a binary variable length Markov chain (VLMC), deriving explicit lower bounds on stationary masses that yield a concrete growth regime (e.g., $m_n\log m_n/n \to 0$) ensuring the CLT. These results provide asymptotic foundations for inference in sparse/partitioned higher-order models; including VLMCs and sparse Markov models (SMMs) where the effective dimensionality grows with the sample size. 
\end{abstract}

\section{Introduction} \label{chap3:sec_introduction}
Large-sample properties of first-order Markov chains are well understood and have wide applications. However, in many applications, the present state depends on the previous $m$ states. Let $\Phi = \{X_t: t = 0,1,2,\ldots\}$ be a sequence of random variables taking values in a finite state space $\Sigma$ of size
$|\Sigma| = d$. For $t \ge m$, if $P(X_{t} = x_{t} \mid X_{t-1} = x_{t-1}, \ldots, X_0 = x_0) = P(X_{t} = x_{t} \mid X_{t-1} = x_{t-1}, \ldots, X_{t-m} = x_{t-m}),$ then $\Phi$ is called a Markov chain of order $m$. Studying a Markov chain of order $m$ is not substantially different from the first-order case, as we can transform it into a first-order Markov chain by taking consecutive $m$-tuples as states. Denote this new chain by $\Phi' = \{\boldsymbol{Y}_t : t = 0,1,2,\ldots\}$, where $\boldsymbol{Y}_t = (X_{t+m-1}, \ldots, X_t)$. The new state space is $\Sigma' = \Sigma^m$, i.e., the set of all possible $m$-tuples from $\Sigma$. The corresponding transition matrix $P'$ has dimension $d^m \times d^m$. In each row of $P'$, at most $d$ entries can be non-zero, since $P(\boldsymbol{Y}_{t+1} = (j_1,\ldots,j_m) \mid \boldsymbol{Y}_t = (i_1,\ldots,i_m)) > 0$ only if $i_k = j_{k+1}$ for $k = 1,\ldots,m-1$. In that case, $
P(\boldsymbol{Y}_{t+1} = (j_1,\ldots,j_m) \mid \boldsymbol{Y}_{t} = (i_1,\ldots,i_m)) = P(X_{t+m} = j_1 \mid X_{t+m-1} = i_1, \ldots, X_t = i_m).$ Hence, the total number of entries to be estimated is $d^m(d-1)$, which grows exponentially with increasing $m$.

To reduce this model complexity, several dimension-reduction techniques have been proposed for higher-order Markov chains. For instance, \cite{raftery1985model} introduced a linear mixture model of the form
\[
P(X_t = j_0 \mid X_{t-1} = j_1, \ldots, X_{t-m} = j_m) = \sum_{i=1}^{m} \lambda_i q_{j_i j_0},
\]
where $\sum_{i=1}^{m} \lambda_i = 1$ and $Q = ((q_{ij}))$ is a known transition matrix on $\Sigma$. The weights $(\lambda_1, \ldots, \lambda_m)$ are estimated via maximum likelihood. Another widely used strategy is the Variable Length Markov Chain (VLMC) model, originally proposed by \cite{rissanen1983universal}, where the context length varies depending on the observed past. \cite{buhlmann1999variable} and \cite{buhlmann2000model} developed model selection methods and studied the asymptotic behavior of VLMCs. In these models, a context tree determines the relevant past suffixes for predicting future states, allowing for substantial reduction in the number of parameters when long dependencies are sparse. Recent work on Bayesian Context Trees (BCT) by \cite{kontoyiannis2020bayesian} and \cite{papageorgiou2022posterior} has provided posterior inference and predictive distributions for discrete-time sequence models.

A more general approach to sparse modeling of higher-order Markov data involves grouping the $m$-tuples in $\Sigma^m$ into partitions such that tuples within the same group share identical transition probabilities. These models are referred to as sparse Markov models (SMMs), first proposed by \cite{garcia2011minimal} and later extended in a Bayesian framework by \cite{jaaskinen2014sparse}. \cite{xiong2016recursive} introduced a recursive partitioning algorithm, while \cite{bennett2023fitting} used Gibbs sampling for posterior estimation. More recently, \cite{majumder2026fitting} developed a convex clustering approach for learning the underlying partition structure in the SMM set-up.

As long as the state space $\Sigma$ is finite, the central limit theorems (CLTs) and other large-sample results for SMMs mirror those for first-order Markov chains, provided certain ergodicity conditions hold. For example, \cite{buhlmann1999variable} derived large-sample properties of VLMC estimators, including bootstrap-based CLTs. However, the scenario becomes more intriguing when the Markov chain order grows with the sequence length, i.e., when the order is a function of $n$, denoted by $m_n$. Proving CLTs in this setup is considerably more challenging. In this paper, we give a central limit theorem for this case.

The rest of the chapter is organized as follows. Section (\ref{chap3:sec_results}) presents our main theoretical result. In Section (\ref{chap3:sec_example}), we illustrate the assumptions and conclusions for the special case of a binary VLMC. Section (\ref{chap3:sec_conclusion}) concludes the chapter with final remarks. 

\section{Central Limit Theorem} \label{chap3:sec_results}
Consider a triangular array of variable-order  Markov chains as follows. The $n^{th}$ chain is presented as $\Phi_n=\{X_0^{(n)},\ldots,X_n^{(n)}\}$, where $X_j^{(n)}\in \Sigma$. Recall that $\Sigma$ is a finite state space with $|\Sigma|=d$. However, the order of the Markov chain $\Phi_n$ (denoted by $m_n$) varies with $n.$ One can easily represent this Markov chain as a first-order Markov chain in the following manner. Let $\boldsymbol{Y_t}^{(n)}=(X_{t+m_n-1}^{(n)},\ldots,X_t^{(n)})$, $t=0,1,\ldots,n-m_n+1$. Then $\Phi_n'=\{\boldsymbol{Y_t}^{(n)}\}$ is a first-order Markov chain with state space $S_n=\Sigma^{m_n}$. We work with the Markov chain $\Phi_n'$ in the sequel.

Our goal is to establish a CLT-type result for a sequence of real-valued functions over $\Sigma$ in this triangular-array set up. Before going into the main result, we define some useful notation. Assume that the chain $\Phi_n'$ is aperiodic and irreducible for each $n$. Let $\pi_n$ be its corresponding stationary probability vector of length $d^{m_n}.$ For a state $\alpha\in S_n$, define $\sigma_\alpha^{(n)}(0) \equiv \inf \{j\ge 0: \boldsymbol{Y_j}^{(n)}=\alpha \}$. Thus $\sigma_\alpha^{(n)}(0)$ is the first hitting time of state $\alpha$ for the chain $\Phi_n'$. In a similar fashion, successively define the $k^{th}$ hitting time of $\alpha$ as follows:
$$
\sigma_\alpha^{(n)}(k)= \inf \{j> \sigma_\alpha^{(n)}(k-1): \boldsymbol{Y_j}^{(n)}=\alpha \}.
$$
Also define 
\begin{equation}\label{defn_l_n}
   \ell_{n,\alpha} \equiv \max \{k: \sigma_\alpha^{(n)}(k) \leq n-m_n+1 \}=\sum_{j=0}^{n-m_n+1}\mathcal{I}( \boldsymbol{Y_j}^{(n)}=\alpha)-1. 
\end{equation}
Then $\ell_{n,\alpha}$ is the number of returns of $\Phi_n'$ to state $\alpha$ after the first hitting time. 

If the order of the Markov chains in each array is fixed, i.e $m_n=m$ for some $m\in \mathbb{N}$, then any state $\alpha\in \Sigma^m$ is recurrent, so that $\ell_{n,\alpha} \to \infty$ as $n\to \infty$. However, in our set-up, $m_n\to \infty$ as $n\to \infty$. Note that for $\alpha \in \Sigma^{m_n}$, $E_{\pi_n}(\ell_{n,\alpha})=(n-m_n+2)\pi_n(\alpha)-1$, where $E_{\pi_n}[.]$ is the expectation under the stationarity assumption. If there exists a sequence of states $\{\alpha_n :\alpha_n \in \Sigma^{m_n}\}$ such that $n\pi_n(\alpha_n)\to \infty$ as $n \to \infty$, then $E_{\pi_n}(\ell_{n,\alpha_n})\to \infty$, i.e. the expected number of returns to $\alpha_n$ diverges. 

Next, we present our main result, which demonstrates the large sample properties of a triangular array of sparse Markov chains.
\begin{theorem}\label{chap3:thm_clt1}
Consider the following triangular array of higher order Markov chains over finite alphabet $\Sigma$ of cardinality $d.$ The $n^{th}$ array $\Phi_n=\{X_{0}^{(n)},\ldots,X_{n}^{(n)}\}$ has order $m_n$. Define $\boldsymbol{Y_t}^{(n)} \equiv (X_{t+m_n-1}^{(n)},\ldots,X_t^{(n)})$, $t=0,1,\ldots,n-m_n+1$. 
Denote the stationary probability vector of the $n^{th}$ chain by $\pi_n$. Suppose the following assumptions hold. 
\begin{itemize}
    \item[(i)] For each $n\in \mathbb{N}$, $\Phi_n'=\{\boldsymbol{Y_t}^{(n)}\}$ is an aperiodic and irreducible Markov chain over $\Sigma^{m_n}$.
    \item[(ii)] $m_n \to \infty$ as $n \to\infty$ with $\frac{m_n}{n}\to 0$.
    \item[(iii)] There exists a sequence of states $\{\alpha_n :\alpha_n \in \Sigma^{m_n}\}$ such that $n\pi_n(\alpha_n)\to \infty$ as $n \to \infty$. 
    \item[(iv)] $Var_{\pi_n}\big(l_{n,\alpha_n}/n\pi_n(\alpha_n)\big)\to 0$ as $n\to \infty$. 
\end{itemize}
Define $E_{\alpha_n}[.]$ as the expectation under the assumption that the initial state of the Markov chain is $\alpha_n$. Consider a sequence of functions $g_n:\Sigma^{m_n}\to \mathbb{R}$, such that 
\begin{itemize}
    \item[(i)] For each $n$, $E_{\alpha_n}\Big[\sum_{j=1}^{\tau_{\alpha_n}}\bar{g}_n(\boldsymbol{Y_j}^{(n)})\Big]^2<\infty$;
    \item[(ii)] $\sup_{n}\left\{\dfrac{E_{\alpha_n}\Big[\sum_{j=1}^{\tau_{\alpha_n}}|\bar{g}_n(\boldsymbol{Y_j}^{(n)})|\Big]^2}{E_{\alpha_n}\Big[\sum_{j=1}^{\tau_{\alpha_n}}\bar{g}_n(\boldsymbol{Y_j}^{(n)})\Big]^2}\right\}<\infty$;
\end{itemize}
where $\bar{g}_n(x)=g_n(x)-E_{\pi_n}(g_n)$.\\ 
Under the above conditions,
$$
\dfrac{1}{\sqrt{n}}\sum_{j=0}^{n-m_n+1}\dfrac{\bar{g}_n(\boldsymbol{Y_j}^{(n)})}{\sqrt{\dfrac{E_{\alpha_n}(\tau_{\alpha_n})}{E_{\alpha_n}\Big[\sum_{j=1}^{\tau_{\alpha_n}}\bar{g}_n(\boldsymbol{Y_j}^{(n)})\Big]^2}}}\xrightarrow{d} \mathcal{N}(0,1)
$$
\end{theorem}

\textbf{Proof:}  From now on, we will denote $\ell_{n,\alpha_n}$ as $\ell_n$ for simplicity. First observe that $E_{\pi_n}(\ell_n)=(n-m_n+2)\pi_n(\alpha_n)-1$, i.e. $E_{\pi_n}\left(\frac{\ell_n}{n\pi_n(\alpha_n)}\right)\to 1$ as $n\to \infty$. Also, by assumption (iv), $Var_{\pi_n}\left(\frac{\ell_n}{n\pi_n(\alpha_n)}\right)\to 0$. Hence $\frac{\ell_n}{n\pi_n(\alpha_n)}\xrightarrow[]{p}1$ as $n\to \infty$, eventually leading to $\ell_n\xrightarrow[]{p} \infty $.\\
For any function $f:\Sigma^{m_n}\to \mathbb{R}$, define
$$
s_j^{(n)}(f)=\sum_{t=\sigma_{\alpha_n}(j)+1}^{\sigma_{\alpha_n}(j+1)}f(\boldsymbol{Y_t}^{(n)}), \quad j=0,1,...,\ell_n-1.
$$
By the strong Markov property, $s_j^{(n)}(g_n)$ are iid random variables with mean
$E_{\alpha_n}\Big[\sum_{j=1}^{\tau_{\alpha_n}}g_n(\boldsymbol{Y_j}^{(n)})\Big]=E_{\alpha_n}\Big[s_0^{(n)}(g_n)\Big]$ and variance $Var_{\alpha_n}\Big[s_0^{(n)}(g_n)\Big]=E_{\alpha_n}\Big[\big(s_0^{(n)}(g_n) \big)^2 \Big]-E^2_{\alpha_n}\Big[s_0^{(n)}(g_n)\Big]$. 
Define $\bar{g}_n(x)=g_n(x)-E_{\pi_n}(g_n(X))$, where $X\sim \pi_n$. For simplicity, let's define $E_{\pi_n}(g_n(X))=E_{\pi_n}\big[g_n\big]$ Note that $E_{\pi_n}\big[g_n\big]=\pi_n(\alpha_n)E_{\alpha_n}\big[\sum_{j=1}^{\tau_{\alpha_n}} g_n(\boldsymbol{Y_j}^{(n)})\big]$. Hence 
$$
\begin{aligned}
E_{\alpha_n}\big[s_j^{(n)}(\bar{g}_n)\big]&=E_{\alpha_n}\Big[\sum_{j=1}^{\tau_{\alpha_n}}\big( g_n(\boldsymbol{Y_j}^{(n)})-E_{\pi_n}(g_n)\big)\Big]\\
&=E_{\alpha_n}\Big[\sum_{j=1}^{\tau_{\alpha_n}} g_n(\boldsymbol{Y_j}^{(n)})-\tau_{\alpha_n}E_{\pi_n}(g_n)\Big]\\
&=E_{\alpha_n}\Big[\sum_{j=1}^{\tau_{\alpha_n}} g_n(\boldsymbol{Y_j}^{(n)})\Big]-E_{\alpha_n}\Big[\tau_{\alpha_n}\Big]E_{\pi_n}(g_n)\\
&=0,\quad \text{since $\pi_n(\alpha_n)=1/E_{\alpha_n}\Big[\tau_{\alpha_n}\Big]$}.
\end{aligned}
$$
The normalized and centralized random variables are defined as follows:
$$
s_j'^{(n)}(\bar{g}_n)=\dfrac{s_j^{(n)}(\bar{g}_n)}{\sqrt{Var_{\alpha_n}\Big[s_0^{(n)}(\bar{g}_n)}\Big]}=\dfrac{s_j^{(n)}(\bar{g}_n)}{\sqrt{E_{\alpha_n}\Big[s_0^{(n)}(\bar{g}_n)\Big]^2}}.
$$
Thus, $s_j'^{(n)}(\bar{g}_n), j=0,1,...,\ell_n-1$ are i.i.d random variables with mean $0$ and variance $1$. Denote $v_n^2=E_{\alpha_n}\Big[s_0^{(n)}(\bar{g}_n)\Big]^2$. Fix $0<\epsilon<1$. Define $\underline{n}=[(1-\epsilon)(n-m_n+2)\pi_n(\alpha_n)]$, $\bar{n}=[(1+\epsilon)(n-m_n+2)\pi_n(\alpha_n)]$, and $n^*=[(n-m_n+2)\pi_n(\alpha_n)]$, where $[\cdot]$ is the greatest integer function.\\
Note that
$$
\begin{aligned}
&\sum_{j=0}^{n-m_n+1}\bar{g}_n(\boldsymbol{Y_j}^{(n)})=\sum_{t=0}^{\sigma_{\alpha_n}(0)}\bar{g}_n(\boldsymbol{Y_t}^{(n)})+\sum_{j=0}^{\ell_n-1}s_j^{(n)}(\bar{g}_n)+\sum_{t=\sigma_{\alpha_n}(\ell_n)+1}^{n-m_n+1}\bar{g}_n(\boldsymbol{Y_t}^{(n)})\\
\implies & \Big\lvert \sum_{j=0}^{n-m_n+1}\bar{g}_n(\boldsymbol{Y_j}^{(n)})- \sum_{j=0}^{\ell_n-1}s_j^{(n)}(\bar{g}_n) \Big\rvert \le  \sum_{t=0}^{\sigma_{\alpha_n}(0)}\Big\lvert\bar{g}_n(\boldsymbol{Y_t}^{(n)})\Big\rvert  + s_{\ell_n}^{(n)}(|\bar{g}_n|)\\
\implies & \Big\lvert \dfrac{1}{\sqrt{n^*}v_n}\sum_{j=0}^{n-m_n+1}\bar{g}_n(\boldsymbol{Y_j}^{(n)})-\dfrac{1}{\sqrt{n^*}} \sum_{j=0}^{\ell_n-1}s_j'^{(n)}(\bar{g}_n) \Big\rvert \le  \dfrac{1}{\sqrt{n^*}v_n}\sum_{t=0}^{\sigma_{\alpha_n}(0)}\Big\lvert\bar{g}_n(\boldsymbol{Y_t}^{(n)})\Big\rvert  +\dfrac{s_{\ell_n}^{(n)}(|\bar{g}_n|)}{\sqrt{n^*}v_n} .
\end{aligned}
$$

Define
\[
Z_{n,j}=\frac{s_j^{(n)}(|\bar{g}_n|)}{v_n}.
\]
By assumption (ii), there exists a constant $C<\infty$ such that
\[
\sup_n E_{\alpha_n}\left[Z_{n,0}^2\right]
=
\sup_n
\frac{
E_{\alpha_n}\left[\left(s_0^{(n)}(|\bar{g}_n|)\right)^2\right]
}{v_n^2}
\le C.
\]
Recall that $\ell_n/n^*\xrightarrow{p}1$. Fix $\eta>0$ and
$0<\delta<1$. Then
\[
\begin{aligned}
P\left(
\frac{Z_{n,\ell_n}}{\sqrt{n^*}}>\eta
\right)
&\le
P\left(
\left|\frac{\ell_n}{n^*}-1\right|>\delta
\right)\\
&\quad+
P\left(
\max_{(1-\delta)n^*\le j\le(1+\delta)n^*}
Z_{n,j}>\eta\sqrt{n^*}
\right).
\end{aligned}
\]
By the union bound and Markov's inequality,
\[
\begin{aligned}
P\left(
\max_{(1-\delta)n^*\le j\le(1+\delta)n^*}
Z_{n,j}>\eta\sqrt{n^*}
\right)
&\le
\sum_{j=\lfloor(1-\delta)n^*\rfloor}^{\lceil(1+\delta)n^*\rceil}
P\left(
Z_{n,j}>\eta\sqrt{n^*}
\right)\\
&\le
\sum_{j=\lfloor(1-\delta)n^*\rfloor}^{\lceil(1+\delta)n^*\rceil}
\frac{E_{\alpha_n}[Z_{n,j}^2]}{\eta^2n^*}\\
&\le
\frac{(2\delta n^*+2)C}{\eta^2n^*}.
\end{aligned}
\]
Consequently,
\[
\limsup_{n\to\infty}
P\left(
\frac{Z_{n,\ell_n}}{\sqrt{n^*}}>\eta
\right)
\le
\frac{2C\delta}{\eta^2}.
\]
Since $\delta>0$ is arbitrary, letting $\delta$ arbitrarily small
\[
\frac{Z_{n,\ell_n}}{\sqrt{n^*}}\xrightarrow{p}0.
\]
Therefore,
\[
\frac{s_{\ell_n}^{(n)}(|\bar{g}_n|)}
{\sqrt{n^*}v_n}
\xrightarrow{p}0.
\]
By similar logic, $$
\dfrac{1}{\sqrt{n^*}v_n}\sum_{t=0}^{\sigma_{\alpha_n}(0)}\Big\lvert\bar{g}_n(\boldsymbol{Y_t}^{(n)})\Big\rvert \xrightarrow[]{p}0.
$$ Combining these results, we obtain
\begin{equation}\label{step1}
    \Big\lvert \dfrac{1}{\sqrt{n^*}v_n}\sum_{j=0}^{n-m_n+1}\bar{g}_n(\boldsymbol{Y_j}^{(n)})-\dfrac{1}{\sqrt{n^*}} \sum_{j=0}^{\ell_n-1}s_j'^{(n)}(\bar{g}_n) \Big\rvert \xrightarrow{p}0.
\end{equation}

Since $\frac{\ell_n}{n\pi_n(\alpha_n)}\xrightarrow[]{p}1$ as $n\to \infty$, $\exists\quad n_0\in \mathbb{N}$ s.t. for $n \ge n_0$
$$
P(\underline{n}\le \ell_n-1 \le \bar{n})\ge 1-\epsilon.
$$
Hence, by Kolmogorov's inequality, for any arbitrary $\beta>0$,
\begin{equation}\label{kolmogorov}
\begin{aligned}[b]
    P \Big(\big\lvert\dfrac{1}{\sqrt{n^*}}\sum_{j=0}^{\ell_n-1}s_j'(\bar{g}_n)-\dfrac{1}{\sqrt{n^*}}\sum_{j=0}^{n^*}s_j'(\bar{g}_n)  \big\rvert > \beta \Big) & \le \epsilon  && + P \Big(\big\lvert \max_{\underline{n}\le l \le n^*}\sum_{j=l}^{n^*}s_j'(\bar{g}_n) \big\rvert > \beta\sqrt{n^*} \Big) \\
    &  && + P \Big(\big\lvert \max_{n^*\le l \le \bar{n}}\sum_{j=n*}^{l}s_j'(\bar{g}_n) \big\rvert > \beta\sqrt{n^*} \Big)\\
    & \le \epsilon &&+\dfrac{2\epsilon\pi_n(\alpha_n)(n-m_n+2)E_{\alpha_n}\Big[s_0^{(n)}(\bar{g}_n)\Big]^2}{\beta^2 n^*} \\
    &\le \epsilon &&+ \dfrac{4\epsilon}{\beta^2} .
\end{aligned}
\end{equation}
From equation (\ref{kolmogorov}), we can conclude that
\begin{equation}
    \Big\lvert\dfrac{1}{\sqrt{n^*}}\sum_{j=0}^{\ell_n-1}s_j'(\bar{g}_n)-\dfrac{1}{\sqrt{n^*}}\sum_{j=0}^{n^*}s_j'(\bar{g}_n) \Big\rvert \xrightarrow{p} 0.
\end{equation}
Combining this with equation (\ref{step1}), we have
\begin{equation}\label{step2}
    \dfrac{1}{\sqrt{n^*}v_n}\sum_{j=0}^{n-m_n+1}\bar{g}_n(\boldsymbol{Y_j}^{(n)})- \dfrac{1}{\sqrt{n^*}}\sum_{j=0}^{n^*}s_j'(\bar{g}_n) \xrightarrow{p} 0.
\end{equation}
Using the CLT for row-wise i.i.d random variables with mean $0$ and variance $1$,
we obtain 
\begin{equation}\label{step3}
\dfrac{1}{\sqrt{n^*}}\sum_{j=0}^{n^*}s_j'(\bar{g}_n) \xrightarrow{d} \mathcal{N}(0,1).
\end{equation}
Hence, from equation (\ref{step2}), we conclude
$$
\begin{aligned}
& \dfrac{1}{\sqrt{n^*}v_n}\sum_{j=0}^{n-m_n+1}\bar{g}_n(\boldsymbol{Y_j}^{(n)})\xrightarrow{d} \mathcal{N}(0,1)\\
\implies & \dfrac{1}{\sqrt{n\pi_n(\alpha_n)}v_n}\sum_{j=0}^{n-m_n+1}\bar{g}_n(\boldsymbol{Y_j}^{(n)})\xrightarrow{d} \mathcal{N}(0,1)\\
\implies & \dfrac{1}{\sqrt{n}}\sum_{j=0}^{n-m_n+1}\dfrac{\bar{g}_n(\boldsymbol{Y_j}^{(n)})}{\sqrt{\dfrac{E_{\alpha_n}(\tau_{\alpha_n})}{E_{\alpha_n}\Big[s_0^{(n)}(\bar{g}_n)\Big]^2}}}\xrightarrow{d} \mathcal{N}(0,1).
\end{aligned}
$$

\section{Example} \label{chap3:sec_example}
In the previous section, we discussed requirements for a large sample CLT. One of the conditions to satisfy the CLT is $n\pi_n(\alpha_n)\to\infty$ as $n\to \infty$. It is easy to see that there is always a state $\alpha_n\in\Sigma^{m_n}$ with $\pi_n(\alpha_n)\ge 1/d^{m_n}$, hence assuming $n/d^{m_n}\to \infty$ will suffice. However, this bound is the most obvious one. If we can make the bound tighter, we can hope for $m_n\to\infty$ at a faster rate. One such example is given below. 

\begin{example}
    Suppose $\Sigma=\{0,1\}$. Consider a VLMC of length $n$ and order $m$ with the contexts $\{0,10,110,\ldots,1^{m-1}0,1^{m}\}$. Hence there are $m+1$ leaves in the context tree. Suppose that for a fixed $m$, $P(0|1^{j}0)=1-p_j$ and $P(1|1^{j}0)=p_j$, $j=0,1,\ldots,m-1$; and $P(0|1^{m})=1-p_m$. In this set-up, we find that
$$
\dfrac{1}{\pi_n(1^{m-1}0)}=\prod_{j=0}^{m-2}\frac{1}{p_j}+\frac{p_{m-1}+1-p_m}{1-p_m}+\sum_{j=1}^{m-2}\prod_{k=j}^{m-2}\frac{1}{p_j}.
$$
\end{example}
\textbf{Proof:} We use recursive relations for stationary probabilities to determine $\pi_n(1^{m-1}0)$. For any history $(i_1,\ldots,i_m)\in\{0,1\}^m$, denote the index by $w(i_1,\ldots,i_m)=1+\sum_{j=0}^{m-1}2^ji_{m-j}$. Then the function $w$ will provide a unique index, ranging from $1$ to $2^m$. For simplicity, we write $\pi_n(j)$ as the stationary probability of the $m$-tuple $(i_1,\ldots,i_m)\in\{0,1\}^m$ having $w(i_1,\ldots,i_m)=j.$ Using the definition of stationary probability, we find that
\begin{equation}\label{eq_relation_1}
    \sum_{j=0}^{2^{m-k}-1}\pi_n(2^kj+2^{k-1})=\sum_{j=0}^{2^{m-k}-1}\pi_n(2^kj+2^{k}-1);\quad k=2,3,\ldots,m.
\end{equation}
Also define the following quantities and obtain the relations
\begin{equation}
\begin{aligned}
        A_{k-1,1}:=\sum_{j=0}^{2^{m-k}-1}\pi_n(2^kj+2^{k-1}-1)=c_{k-1}\sum_{j=0}^{2^{m-k}-1}\pi_n(2^kj+2^{k-1})=:c_{k-1}A_{k-1,2}
\end{aligned}
\end{equation}
and $c_{k-1}=(1-p_{k-2})/p_{k-2}$ for $k=2,3,\ldots,m$ and $c_m=(1-p_{m-1})/p_m$. Define $A_{m,1}=\pi_n(2^m-1)$, $A_{m,2}=\pi_n(2^m)$. Now, observe that
$$
\begin{aligned}
A_{m-2,1}&=c_{m-2}A_{m-2,2}=c_{m-2}(A_{m-1,1}+A_{m,1})\\
&=c_{m-2}(1+c_{m-1})A_{m,1};\\
A_{m-3,1}&=c_{m-3}(1+c_{m-2})(1+c_{m-1})A_{m,1};\\
\ldots & \\
A_{1,1}&=c_1(1+c_2)\ldots(1+c_{m-1})A_{m,1}
\end{aligned}
$$
and
$$
\begin{aligned}
A_{1,2}&=(1+c_2)\ldots(1+c_{m-1})A_{m,1}\\
&=c_{m-2}(1+c_{m-1})A_{m,1};\\
A_{2,2}&=(1+c_3)\ldots(1+c_{m-1})A_{m,1};\\
\ldots & \\
A_{m-2,2}&=(1+c_{m-1})A_{m,1};\\
A_{m-1,2} &= A_{m,1}.
\end{aligned}
$$
Combining these relations, we have
$$
A_{m,1}=\pi_n(2^m-1)=\dfrac{1}{\prod_{i=1}^{m-1}a_i+a_m+\sum_{i=2}^{m-1}a_ia_{i+1}\ldots a_{m-1}},
$$
where $a_i=1+c_i$. Replacing the values of $c_i$, we find that
$$
\dfrac{1}{\pi_n(1^{m-1}0)}=\prod_{j=0}^{m-2}\frac{1}{p_j}+\frac{p_{m-1}+1-p_m}{1-p_m}+\sum_{j=1}^{m-2}\prod_{k=j}^{m-2}\frac{1}{p_j}:=q(\boldsymbol{p},m).
$$
If the transition probabilities satisfy $q(\boldsymbol{p},m_n)=o(n)$, we can claim $n\pi_n(1^{m-1}0)\to\infty$..
\begin{corollary}
Suppose $p_j=(j+1)/(j+2)$. Then $n\pi_n(1^{m_n-1}0)\to\infty$ if $m_n\log(m_n)/n \to 0$. 
\end{corollary}
\textbf{Proof:} $q(p,m_n)\le 2m_n+m_n\sum_{j=1}^{m_n-2}\frac{1}{j+1}=\mathcal{O}(m_n\log(m_n))$. Hence, the required condition simplifies to $m_n\log(m_n)/n \to 0$ as $n\to \infty$.
\section{Conclusion}\label{chap3:sec_conclusion}
The theoretical results given guarantee central limit theorems for certain functionals of a higher order Markov chain, whose order changes with the length of the chain.  We also simplify the conditions of the theorem (\ref{chap3:thm_clt1}) for a particular type of binary variable length Markov chain and determined how the order $m_n$ varies with $n$. Similar reductions may be possible for Sparse Markov Models (SMMs) by exploiting their underlying partition structure. Investigating such conditions for SMMs and other sparse higher-order Markov models is an interesting direction for future work.

\bibliographystyle{plainnat}
\bibliography{reference}

@STRING{JES       = "J.\ Earth Syst.\ Sci."}

@article{raftery1985model,
  title={A model for high-order \uppercase{M}arkov chains},
  author={Raftery, Adrian E},
  journal={Journal of the Royal Statistical Society: Series B (Methodological)},
  volume={47},
  number={3},
  pages={528--539},
  year={1985},
  publisher={Wiley Online Library}
}

@article{jaaskinen2014sparse,
  title={Sparse \uppercase{M}arkov chains for sequence data},
  author={J{\"a}{\"a}skinen, V{\"a}in{\"o} and Xiong, Jie and Corander, Jukka and Koski, Timo},
  journal={Scandinavian Journal of Statistics},
  volume={41},
  number={3},
  pages={639--655},
  year={2014},
  publisher={Wiley Online Library}
}

@article{xiong2016recursive,
  title={Recursive learning for sparse \uppercase{M}arkov models},
  author={Xiong, Jie and J{\"a}{\"a}skinen, V{\"a}in{\"o} and Corander, Jukka and others},
  journal={Bayesian analysis},
  volume={11},
  number={1},
  pages={247--263},
  year={2016},
  publisher={International Society for Bayesian Analysis}
}

@article{buhlmann1999variable,
  title={Variable length \uppercase{M}arkov chains},
  author={B{\"u}hlmann, Peter and Wyner, Abraham J},
  journal={The Annals of Statistics},
  volume={27},
  number={2},
  pages={480--513},
  year={1999},
  publisher={Institute of Mathematical Statistics}
}

@article{buhlmann2000model,
  title={Model selection for variable length \uppercase{M}arkov chains and tuning the context algorithm},
  author={B{\"u}hlmann, Peter},
  journal={Annals of the Institute of Statistical Mathematics},
  volume={52},
  number={2},
  pages={287--315},
  year={2000},
  publisher={Springer}
}

@article{rissanen1983universal,
  title={A universal prior for integers and estimation by minimum description length},
  author={Rissanen, Jorma},
  journal={The Annals of statistics},
  volume={11},
  number={2},
  pages={416--431},
  year={1983},
  publisher={Institute of Mathematical Statistics}
}

@inproceedings{garcia2011minimal,
  title={Minimal \uppercase{M}arkov models},
  author={Garc{\i}a, Jes{\'u}s E and Gonz{\'a}lez-L{\'o}pez, Ver{\'o}nica A and de Holanda, Rua Sergio Buarque and Geraldo, Cidade Universit{\'a}ria-Barao},
  booktitle={Fourth Workshop on Information Theoretic Methods in Science and Engineering},
  pages={25},
  year={2011}
}

@article{papageorgiou2022posterior,
  title={Posterior representations for Bayesian context trees: Sampling, estimation and convergence},
  author={Papageorgiou, Ioannis and Kontoyiannis, Ioannis},
  journal={arXiv preprint arXiv:2202.02239},
  year={2022}
}

@article{kontoyiannis2020bayesian,
  title={Bayesian context trees: Modelling and exact inference for discrete time series},
  author={Kontoyiannis, Ioannis and Mertzanis, Lambros and Panotopoulou, Athina and Papageorgiou, Ioannis and Skoularidou, Maria},
  journal={arXiv preprint arXiv:2007.14900},
  year={2020}
}

@article{bennett2023fitting,
  title={Fitting sparse {Markov} models through a collapsed Gibbs sampler},
  author={Bennett, Iris and Martin, Donald EK and Lahiri, Soumendra Nath},
  journal={Computational Statistics},
  volume={38},
  number={4},
  pages={1977--1994},
  year={2023},
  publisher={Springer}
}

@article{majumder2026fitting,
  title={Fitting sparse {Markov} models to categorical time series using convex clustering},
  author={Majumder, Tuhin and Lahiri, Soumendra N and Martin, Donald EK},
  journal={Journal of Computational and Graphical Statistics},
  pages={1--13},
  year={2026},
  publisher={Taylor \& Francis}
}
\end{document}